\documentclass[11pt]{amsart}
\usepackage{graphicx}
\usepackage{latexsym}
\usepackage{amsfonts,amsmath,amssymb}
\newtheorem{theorem}{Theorem}[section]
\newtheorem{proposition}[theorem]{Proposition}
\newtheorem{lemma}[theorem]{Lemma}

\newtheorem{corollary}[theorem]{Corollary}

\def\si{\par\smallskip\noindent}
\def\bi{\par\bigskip\noindent}

\def\ex{\textup{E\/}}

\def\eps{\varepsilon}

\def\a{\alpha}
\def\be{\beta}

\def\ga{\gamma}
\def\part{\partial}
\def\Cal{\mathcal}

\usepackage{xcolor}

\newcommand{\beq}{\begin{equation}}
\newcommand{\eeq}{\end{equation}}

\theoremstyle{remark}

\numberwithin{equation}{section}
\date{\today}
\begin{document}

\title[Stable permutations]{One-sided version of Gale-Shapley proposal algorithm and its likely behavior under random
preferences.}

\author{Boris Pittel}
\address{Department of Mathematics, The Ohio State University, Columbus, Ohio 43210, USA}
\email{bgp@math.ohio-state.edu}

\maketitle

\begin{abstract}  
For a two-sided ($n$ men/$n$ women) stable matching problem) Gale and Shapley studied a proposal algorithm (men propose/women
select, or the other way around), that determines a matching, not blocked by any unmatched pair. Irving used this algorithm as
a first phase of his algorithm for one-sided (stable roommates) matching problem with $n$ agents. We analyze a fully extended version of Irving's
proposal algorithm that runs all the way until either each agent holds a proposal or an agent gets rejected by everybody on the agent's
preference list.  It is shown that the terminal, directed, partnerships form a stable permutation with
matched pairs remaining matched  in any other stable permutation. A likely behavior of the proposal algorithm is studied under
assumption that all $n$ rankings are independently uniform. It is proved that with high probability (w.h.p.) every agent has a partner, and
that both the number of agents in cycles of length $\ge 3$ and the total number of stable
matchings are bounded in probability. 
W.h.p. the total number of proposals is asymptotic to $0.5 n^{3/2}$.
\end{abstract}

\section{Introduction and main results.} In 1962 Gale and Shapley \cite{GalSha} introduced and analyzed a game-theoretical model; 
it transcends the issues of college admissions and stable marriages they used to illuminate the ideas. There are two 
equinumerous sets of agents, $n$ ``men'' and $n$ ``women'', each agent ranking strictly the agents on the other side as potential 
``marriage partners''. The problem is to find a stable matching $M$ of two sides, i.e. a matching that cannot be destabilized by any unmatched
pair: formally, no unmatched pair $(m,w)$ is such that $m$ prefers $w$ to his partner $M(m)$ in $M$, and $w$ prefers $m$ to her partner
$M(w)$ in $M$. They discovered a remarkable ``proposal'' algorithm that delivers such a matching. In round $1$ men  select  (``propose'' to) 
their first choice women, and each selected woman provisionally puts on hold the best among the proposers, rejecting the others, if there
are any.  Recursively, at each round the rejected men propose to their next best choices, and selected women reject all men except the
currently best one. The process terminates when all women hold a proposal. Gale and Shapley proved that the terminal matching is stable,
and is man-optimal, meaning that each man gets the best stable partner. McVitie and Wilson \cite{McVWil} proved later that  this matching is woman-pessimal, i.e. each woman gets her worst stable husband. The situation is completely reversed when women propose, and each man rejects all but the currently best woman proposer. To quote from Gusfield and Irving \cite{GusIrv}, The analysis in \cite{McVWil} was based on an ``alternative formulation of algorithm, in which the men initiate their proposal sequences in a fixed order, and every rejection causes the rejected
man to make his next proposal immediately''.

These numbers-free combinatorial algorithms definitely called for an average case analysis, and it was Wilson \cite{Wil} who demonstrated
that the expected number of proposals for a uniformly random problem instance is bounded above by $nH_n$, $H_n=
1/1+\cdots+1/n\sim \log n$. He accomplished this feat by using an ingenious reduction of the proposal process for a classic coupon-collector
problem. 

Later Knuth \cite{Knu0} (\cite{Knu}) undertook a systematic study of the problem. In particular, he found a better upper bound $(n-1)H_n$ and a matching lower bound $nH_n-O(\log^4 n)$ for the expected number of proposals. Knuth posed a problem of estimating
$\Bbb E[S_n]$, $S_n$ being the total number of stable matchings, suggesting that this might be done by means of his integral formula
for the probability that a given matching is stable. This was done in \cite{Pit1} where we proved that $\Bbb E[S_n]\sim e^{-1} n\log n$. 
Subsequently Lennon and Pittel \cite{LenPit} extended the techniques in \cite{Pit1} to
show that $\ex[S_n^2]\sim (e^{-2}+0.5e^{-3})(n\log n)^2$. These two estimates together implied that $S_n$ is of order $n\log n$ with probability $0.84$, at least.

Extending the technique of integral formulas, we proved in \cite{Pit2} a ``law of hyperbola'': it states that with high probability (w.h.p.) (i.e. with probability $\ge 1-o(1)$)  the product of the total men's rank and the total women's rank in every stable matching is asymptotic to $n^3$.
In particular, in the men-optimal matching the women's rank is asymptotic to $n\log n$, while the men's rank is asymptotic to $n^2/\log n$,
with the bounds swapped for the women-optimal matching. Using this ``conservation law'' we showed that for the ``minimum-regret'' 
stable matching--the one that minimizes the largest rank of a partner--w.h.p. the men's rank and the women's rank are each asymptotic to $n^{3/2}$, or equivalently the average spouse rank is asymptotic to $n^{1/2}$. In fact, the worst spouse's rank
in this stable matching w.h.p. is of order $n^{1/2}\log n$, an evidence of how well balanced this matching is.

In a recent breakthrough paper Ashlagi, Kanoria and Leshno \cite{Ash} analyzed a random stable matching problem with {\it unequal\/}
numbers $n_1$ and $n_2$ of men and women, $n_2>n_1$ say. They discovered that the likely rank
of men and the likely rank of women do not depend, asymptotically, on which side proposes: e. g. they are $\sim n^2/\log n$ and $\sim n\log n$ respectively, for $n_1=n$ and $n_2=n+1$. This abrupt change of asymptotic behavior
is in contrast to a more moderate, but still sizable, change for the expected number of stable matchings: relative to the balanced case $n_1=n_2=n$, it falls down by a factor $\log^2 n$.

Viewed more broadly, the stable marriage problem is a special case of the one-sided stable matching problem (``roommates problem'') in which every agent from a group
of even cardinality $n$ ranks some of the other agents as potential partners. A matching is a partition of the set $[n]$ into $n/2$ pairs, and it is stable if no two unmatched agents prefer each other to their partners. Since
\cite{GalSha} it was known that not every problem instance has a solution, and Knuth \cite{Knu} asked whether there is a polynomial
time algorithm that finds a stable matching whenever it exists. Irving \cite{Irv0} answered Knuth's question positively by constructing such
an algorithm: its running time is of order $O(m)$ where $m$ is the total number of entries that identify acceptable partners in the
agents' preference lists; see also \cite{GusIrv}. We proved \cite{Pit2.4} that (for the uniformly random problem instance with complete preference
lists) the likely number of basic steps in Irving's algorithm, i.e. proposals in phase I and ``rotations'' in phase II, is of order $n\log n$. In a closely related paper \cite{Pit2.5} we showed that the expected number of stable matchings converges to $e^{1/2}$, in a sharp contrast to $e^{-1}n\log n$, the expected number of stable marriages for the two-sided matching problem.

To quote from Manlove \cite{Man}, ``considering the progress that has been made on SR [stable roommates] after $1989$, a key landmark is the work of Tan (and Hsueh)'' \cite{Tan1}, \cite{Tan2}, \cite{Tan3} ``on stable partitions. This structure...is present in every SR instance, and its existence is strong compensation for the fact that a stable matching need not exist''. Tan's  partition is
a partition of $[n]$, with $n$ not necessarily even, into a union of disjoint directed cycles, i.e. a permutation of $[n]$. 
Each agent's preference list is enlarged by adding one more entry at the end for the least favored option ``no partner''.  A permutation $\Pi$ is stable if {\bf (1)\/} each agent $i$ prefers $\Pi(i)$ (``successor'' of $i$) to $\Pi^{-1}(i)$ (``predecessor''
of $i$); {\bf (2)\/} for all $i,\,j\in [n]$, if $i$ prefers $j$ to $i$'s predecessor, then $j$ prefers $j$'s predecessor to $i$. Tan and Hsueh
\cite{Tan3} described a recursive algorithm that allows to find a stable partition of $[n]$ from a ``reduced'' stable partition of $[n-1]$
by application of a proposal sequence reminiscent of Gale and Shapley's algorithm for stable marriages and Irving's algorithm for the roommates problem. 

Irving and Pittel \cite{IrvPit} used this algorithm to prove that, for $n$ even, the probability that
there is a stable matching (i.e. a stable partition with all $n/2$ cycles of length $2$) is at most $0.5 e^{1/2}$. Extensive numerical experiments compelled Mertens \cite{Mer} to conjecture that this probability approaches zero as $n^{-1/4}$. The rigorous estimates in
\cite{Pit2.5} showed that this probability cannot approach zero faster than $n^{-1/2}$. In \cite{Pit5} we continued analysis of likely behavior of Tan's stable partitions. We showed that the expected number of reduced  stable partitions -- i.e. those consisting only
of matched pairs and odd cycles, the latter necessarily common to all stable partitions -- and the expected number of odd cycles
are $\sim \frac{\Gamma(1/4)}{\sqrt{\pi e}\, 2^{1/4}} n^{1/4}$ and $\lesssim   \frac{\Gamma(1/4)}{4\sqrt{\pi e}\, 2^{1/4}} n^{1/4}$
respectively. The $\log n$ factor aside, the estimates grow polynomially with $n$. 

We combined these estimates, and 
analysis of the variance of the number of reduced stable partitions with Tan and Hsueh's results to prove: {\bf (1)\/} the likely largest size of an
{\it internally\/} stable matching is at least  $n - n^{1/4-o(1)}$; {\bf (2)\/} for $n$ even, w.h.p. there  exists a complete matching
blocked by, at most, $n^{3/4}+o(1)$ unmatched pairs. The fractional powers of $n$ in {\bf (1,2)\/} can possibly be decreased, but it seems doubtful that $n^{1/4}$ and/or $n^{3/4}$ can be replaced by, say, $\log^a n$.

In this paper we study a one-sided proposal algorithm, which is the exact counterpart of the sequential McVitie-Wilson algorithm
for stable marriages.

\begin{theorem}\label{thmA} {\bf (1)\/} The terminal directed partnerships form a permutation
$\Pi$ which satisfies the condition
\begin{equation*} 
i\text{ prefers }\Pi(j)\text{ to }\Pi(i) \Longrightarrow \Pi(j)\text{ prefers }j\text{ to } i;\\
\end{equation*}
{\bf (2)\/} Calling such permutations stable, we have: $\Pi$ is stable only if $\Pi^{-1}$ is.

\noindent {\bf (3)\/} For each agent $i$, the terminal successor $\Pi(i)$ is the best stable successor, and the terminal
predecessor $\Pi^{-1}(i)$ is
the worst stable predecessor. Consequently all possible executions of the proposal algorithm yield the same stable permutation $\Pi_0$,
such that each $i$ prefers $\Pi_0(i)$ to $\Pi_0^{-1}(i)$. All pairs of agents matched in $\Pi_0$ remain matched in all other stable
permutations. Either the stable $\Pi$'s are all fixed-point-free, or they share a single fixed point.
\end{theorem}

We study a likely behavior of the proposal algorithm and the stable permutations under the assumption that all $n$ rankings (i.e. preference lists) are independently
uniform. Here is a summary of our probabilistic claims.

{\bf Definition.\/} A permutation $\Pi$ is called $\Pi_0$-like if it is stable, fixed-point-free and such that each $i$ prefers $\Pi(i)$ to
$\Pi^{-1}(i)$.

\begin{theorem}\label{thmB}  For a stable $\Pi$, let $M(\Pi)$ stand for the total number of matched pairs, and let $R_s(\Pi)$ and $R_p(\Pi)$ stand for the total rank of successors and the total rank of predecessors. {\bf (a)\/}
\begin{align*}
&\Bbb P\bigl(\Pi\!\text{'s have no fixed point}\bigr)\ge 1-\exp\bigl(-n^{1/2+o(1)}\bigr);\\
&\Bbb E\Bigl[\min_{\Pi}M(\Pi)\Bigr]=n/2-O(1);\\
&\Bbb E\bigl[\#\!\text{ of }\Pi_0\!\text{-like permutations}\bigr]=O(1).
\end{align*}
{\bf (b)\/} W.h.p.  $R_s(\Pi),\,R_p(\Pi)= 0.5 n^{3/2}(1+O(n^{-1/2}\log n))$ for all $\Pi\!$'s.  
\end{theorem}
\noindent Thus the newly defined stable permutations are likely to consist of about $n/2$ matched pairs and possibly some cycles of length $3$
or more, that have a bounded expected total length. W.h.p.  the total number of proposals in the algorithm is $0.5 n^{3/2}(1+O(n^{-1/2}\log n))$. To compare, for Tan's stable permutations the bounds are (\cite{Pit5})
\[
\Bbb E\Bigl[\min\limits_{\Pi} M(\Pi)\Bigr]\le n/2 - c_1 n^{1/4},\quad \Bbb E\bigl[\#\!\text{ of }\Pi\text{'s}\bigr]\sim c_2 n^{1/4}.
\]

\section{Proposal algorithm and stable permutations.} We have a set of $n$ agents, each {\it strictly\/} ranking the remaining $n-1$ agents, and we interpret
an agent's ranking of other agents as a preference list for a social partner. Assume  that each agent, free to propose, may propose to
one  agent only. By a common convention, we increase the length of each agent's  preference list to $n$, reserving the last slot
for the least favored option ``agent has no partner''. Here is a sequential proposal algorithm for determination of a system of pairwise, directed partnerships,
naturally aligned with individual agents' preferences. This algorithm is a counterpart of the well-known bipartite proposal algorithm
(McVitie, Wilson) that finds a stable matching between two sides (``men'' and ``women'').

Order agents arbitrarily. At every step we have a set $P$ of agents (``predecessors''), each being put on hold by (attached to) a single agent  from  a equinumerous set $S$ of agents (``successors'').  There is also a set $U$ of unattached agents, 
each not yet rejected by all other agents. If $U\neq \emptyset$,  an unattached agent $u$ proposes to the $u$'s best choice agent $v$ among those who haven't rejected $u$ earlier.  $v$ accepts the proposal if either $v\notin S$ or if $v\in S$, but $v$ prefers $u$ to its current predecessor $w$. $w$ joins the set of unattached agents, provided that there remain
agents who haven't rejected $w$ already. Otherwise $v$ rejects $u$ and $u$ rejoins the set of unattached agents, if there are still agents $u$ hasn't proposed to. The process terminates as soon as $U=\emptyset$.

Let $P_t$ and $S_t$ denote the terminal predecessor set and the terminal successor set, respectively.
\begin{lemma}\label{lem1} $S_t=P_t$ and $|S_t|\in \{n-1,n\}$.
\end{lemma}
\begin{proof} First of all, $|S_t|=|P_t|$, since the algorithm determines not only $S_t$ and $P_t$, but also a bijective
mapping from $P_t$ to $S_t$.
Also, once an agent $v$ receives a proposal, it holds a proposal afterward, so that $v\in S_t$. Suppose $s\in S_t\setminus P_t$, so that 
$s$ has been rejected
by all other $n-1$ agents.  Then those agents are all terminal successors, i.e.  $|S_t|=1+(n-1)=n$, implying that $|S_t|=|P_t|=n$; contradiction. Therefore $S_t\setminus P_t=\emptyset$, meaning that every terminal successor is a terminal predecessor. So, using $|S_t|=|P_t|$, we obtain $S_t=P_t$.

Further, $[n]\setminus S_t=n\setminus P_t$ is the set of agents $s$ such that $s$  has been rejected by all $v\neq s$. Such a set may contain at most one agent.
\end{proof}

The bijection is a fixed-point free permutation of $S_t=P_t$. If $|S_t|=n-1$, we add the single outsider as a cycle of length $1$, and obtain a permutation, $\Pi_0$, of $[n]$ with a single fixed point.

{\bf Definition.\/}  We call a permutation $\Pi$ of $[n]$ stable if there is no pair $i,\,j$, $(i\neq j,\,\Pi(j))$,  
such that $i$ prefers $\Pi(j)$  to $\Pi(i)$ and $\Pi(j)$ prefers $i$ to $j$. 

{\bf Notes.\/} {\bf (a)\/} If $i$ is a fixed point of a stable $\Pi$ then for each $j\neq i$, $\Pi(j)$ prefers $j$ to $i$, whence $\Pi(j)\neq j$
and therefore  $i$ is unique. {\bf (b)\/} If $n$ is even and $\Pi$ is a permutation with $n/2$ cycles of length $2$, then $\Pi$ is stable if and only if the $n/2$ pairs of agents form a stable matching. {\bf (c)\/} Let $Q=\Pi^{-1}$. Observe that $j\neq i, Q(i)$ is equivalent to 
$\Pi(Q(j))\neq Q(i), \Pi(Q(i))$. Substituting $Q(i):=\Pi^{-1}(i)$ and $Q(j):=\Pi^{-1}(j)$ instead of $i$ and $j$ into the stability condition for $\Pi$, we obtain: there is no pair $i,\,j$ , $(j\neq i,\,Q(i))$, such that $Q(i)=\Pi^{-1}(i)$ prefers $\Pi(Q(j))=j$ to $\Pi(Q(i))=i$ and $j$ prefers $Q(i)=\Pi^{-1}(i)$ to $Q(j)=\Pi^{-1}(j)$.  And this is the stability condition for $Q=\Pi^{-1}$.
So $\Pi$ is stable if and only if $\Pi^{-1}$ is.

\begin{lemma}\label{lem2} Both $\Pi_0$ and $\Pi_0^{-1}$ are stable.
\end{lemma}
\begin{proof} Suppose $i$ prefers $\Pi_0(j)$ to $\Pi_0(i)$. Then agent $i$, before proposing to $\Pi_0(i)$, proposed to $\Pi_0(j)$ and subsequently was rejected by $\Pi_0(j)$. $j$ is the terminal  predecessor of $\Pi_0(j)$, whence $j$ is the $\Pi_0(j)$'s favorite among all
agents who proposed to $\Pi_0(j)$. So $\Pi_0(j)$ prefers $j$ to $i$.
\end{proof}

{\bf Note.\/} The stability condition arose from contemplating which changes of two agents' preferences
might have resulted
in a  different terminal permutation. The definition of a stable permutation differs from the one suggested by Tan: $\Pi$ is stable if {\bf (a)\/} there is no pair $(i\neq j)$ such that  $i$ prefers 
$j$ to $\Pi^{-1}(i)$ and $j$ prefers $i$ to $\Pi^{-1}(j)$, and {\bf (b)\/} each $i$ prefers $\Pi(i)$ to $\Pi^{-1}(i)$; see \cite {Tan1}, \cite{Tan2} and \cite{Tan3}. 

To compare: applying our stability condition to $j=\Pi^{-2}(i)\neq i$, we obtain that for each
cycle $C$, $|C|\ge 3$, of a stable $\Pi$, there are only two options: either $i$ prefers $\Pi(i)$ to $\Pi^{-1}(i)$
for all $i\in C$, or $i$ prefers $\Pi^{-1}(i)$ to $\Pi(i)$ for all $i\in C$. 


\begin{lemma}\label{lem3} All possible executions of the proposal algorithm yield the same stable permutation $\Pi_0$, and  for each agent $i$, $\Pi_0(i)$ is the best of the successor-partners $\Pi(i)$ the agent $i$ can have in any stable permutation $\Pi$.
\end{lemma} 

\begin{proof} Let $E$ be an arbitrary execution of the proposal algorithm, and let $\Pi$ be the corresponding stable permutation. 



We need to show that during $E$ no agent rejects its stable predecessor in any stable permutation $\Pi'$. Suppose this is not true, and consider
the first occasion when an agent  $j'$ rejects its stable predecessor, i.e. an agent $i=(\Pi')^{-1}(j')$ for some stable $\Pi'$.

There are two ways this can happen.  

(1)  $i$ proposed to $j'$ when $j'$ already held  proposal from an agent $i'$, preferred by $j'$ to $i$. 

(2) $j'$ already held proposal from $i$ when a better $i'$ proposed to $j'$.
\si
\noindent In either case, if $i'$ prefers $\Pi'(i')$ to $j'=\Pi'(i)$ then, prior to proposing to $j'$, $i'$ must have proposed to, and then must have been rejected by, its stable successor  $\Pi'(i')$.  Contradiction of ``the first occasion'' assumption. Therefore $i'$ prefers $j'=\Pi'(i)$ to $\Pi'(i')$. Since $\Pi'$ is
stable, we obtain that $j'$ prefers $i$ to $i'$. Contradiction of ``$j'$ rejects $i$ in favor of $i'$'' supposition.

So {\it independently\/} of $E$, $\Pi(i)$ is the best stable successor for each $i\in [n]$.
\end{proof}

Combining Lemma \ref{lem2} and Lemma \ref{lem3} we get
\begin{corollary}\label{1} Each $i\in [n]$ prefers $\Pi_0(i)$ to $\Pi_0^{-1}(i)$. Consequently: {\bf (a)\/}  
a fixed point of $\Pi_0$, if there is one, is a fixed point of every other stable permutation $\Pi$; {\bf (b)\/} if $i$ and $j$ form a match
in $\Pi_0$, i.e. $j=\Pi_0(i),\,i=\Pi_0(j)$, then $i$ and $j$ are the only stable partners of each other, so $i$ and $j$ form a match
in any other stable permutation $\Pi$. That is, the set of fixed pairs is the set of matches in $\Pi_0$.
\end{corollary}  




\begin{lemma}\label{lem4} For each agent $i\in [n]$, $\Pi_0^{-1}(i)$ is the worst predecessor that $i$ can have in any 
stable permutation $\Pi$.
\end{lemma}

\begin{proof} Suppose that there is an agent $\a$ and a stable permutation $\Pi_1$ such that $\a$ (strictly) prefers $\Pi_0^{-1}(\a)$
to $\Pi_1^{-1}(\a)$, or setting $i=\Pi_0^{-1}(\a)$, $j=\Pi_1^{-1}(\a)$, that $\Pi_1(j)$ prefers $i$ to $j$. Since $\Pi_1$
is stable, we then have that $i$ prefers $\Pi_1(i)$ to $\Pi_1(j)$, or equivalently that $\Pi_0^{-1}(\a)$ prefers $\Pi_1(\Pi_0^{-1}(\a))$ to $\a$, or setting $\be=\Pi_0^{-1}(\a)$, that $\be$ prefers $\Pi_1(\beta)$ to $\Pi_0(\be)$. This is in contradiction
of Lemma \ref{lem3}.
\end{proof} 

\begin{corollary}\label{1.5} If $\Pi_0$ is fixed point--free, then so is every other stable $\Pi$.
\end{corollary}
\begin{proof} If a stable $\Pi$ has a fixed point $i$, then $\Pi^{-1}(i)=i$ is the worst stable predecessor of $i$, whence $i=\Pi_0^{-1}(i)$ as well,
meaning that $i$ is a fixed point of $\Pi_0$ also.
\end{proof}

{\bf Note.\/}  Lemma \ref{lem3} and Lemma \ref{lem4} are non-bipartite counterparts  of Theorem 1.2.2 and Theorem 1.2.3 (Gusfield and Irving \cite{GusIrv}) for the bipartite
stable matchings: ``Men'' propose to ``women'', and at each step a woman selects a favorite from at most two suitors, the current proponent and a man, if there is one, that she put on hold earlier.

\section{Random stable partitions.} To generate a uniformly random instance $I_n$ of the
$n$ agents' preference lists we introduce an array of the independent random variables $X_{i,j}\,(1\le i\neq j\le n)$, each
distributed uniformly on $[0,1]$. We assume that each agent $i$ ranks the remaining agents $j\neq i$ in increasing order of the variables $X_{i,j}$. Such an ordering is uniform for every $i$, and the orderings by $n$ agents are all independent.
\subsection{A fixed point is very unlikely.} 
\begin{theorem}\label{thm1} $\Bbb P_n:=\Bbb P(\Pi_0\text{ has a fixed point})\le \exp\bigl(-n^{1/2-o(1)}\bigr)$.
\end{theorem}
\begin{proof} Our argument is based on probabilistic analysis of the proposal algorithm, rather than the $\{X_{i,j}\}$-based model, which will be heavily relied upon later. We  use a so-called principle of deferred decisions (\cite{Knu0}, \cite{Knu},\cite{KnuMotPit}).
It is postulated that the random preference system is not given in advance,
but rather unfolds step by step, to the extent necessary for a full run of the proposal algorithm. At each step an agent $i$, whose turn it is to propose, proposes
to a $j\neq i$ agent chosen uniformly at random among the agents $i$ has not proposed to
so far. If it is the $k$-th proposal to $j$  then $i$'s rank (relative to the group of $k$ suitors) is distributed uniformly at random on the set $\{1,\dots k\}$.
In particular, $i$'s proposal is the best so far with probability $1/k$, in which case $j$ puts $i$ on hold and, if $k>1$, rejects the current suitor.
On the event $A_i:=\{i\text{ is a fixed point of }\Pi_0\}$ the process stops when $i$ gets rejected by the last of the other $(n-1)$ agents, who have formed a fixed-point-free permutation on $[n]\setminus\{i\}$ without making a single proposal to $i$. By the
union bound and symmetry, we have $\Bbb P_n\le n\Bbb P(A_n)$, so that we need to show that  $\Bbb P(A_n)\le 
\exp\bigl(-n^{1/2-o(1)}\bigr)$.

For the proof we adapt Wilson's idea \cite{Wil} and assume that as proposers the agents are ``amnesiacs'', while
as a proposee each agent $j$ keeps a continuously updated record of all $j$'s {\it distinct\/} proposers so far. That is, a current proposer $i$
selects an agent $j\neq i$ uniformly at random among {\it all\/} $(n-1)$ agents, and $i$ is put on hold -- while the current suitor,  if there is
one, is rejected -- if and only if $i$ is the best potential partner so far for $j$. Assuming that proposal protocol is ``last rejected/first to propose'',
we see that deleting redundant proposals we get back to a genuine proposal algorithm with the uniformly random selections among agents
still accessible to the current proposer. Notice that we are still free to pick an order in which the agents make their first proposals.  To
bound $\Bbb P(A_n)$ we may and will assume that agent $n$ is the last to make the first proposal. In that case, on the event $A_n$, 
Wilson's process  evolves exclusively on $[n-1]$, with each of the agents in $[n-1]$ eventually receiving at least one proposal, and preferring their terminal proposers to agent $n$. 

Wilson used this device to show that the total number of proposals in the random instance of the $n\times n$ stable marriage problem is below $D_n$, the number of draws in a classic coupon-collector problem, thus showing  that the expected number of proposals is below $\Bbb E[D_n]=nH_n\sim n\log n$, $(H_n=1/1+\cdots+1/n)$. For our proof we need to know that with high probability (w.h.p.) $D_n\ge N_n:=\lfloor(1-\eps)n\log n\rfloor$. Observe that $D_n<d$ means that in the 
$d$ balls/ $n$ boxes uniform allocation scheme all $n$ occupancy numbers are positive. Since the occupancy numbers are negatively
associated (Dubhashi and Ranjan \cite{DubRan}) we see that
\begin{align*}
&\qquad\quad\Bbb P(D_n<N_n)\le \Bigl[1-(1-n^{-1})^N\Bigr]^n\le \exp\Bigl[-n(1-n^{-1})^N\Bigr]\\
&=\exp\Bigl[-n\exp\bigl(N\log(1-n^{-1})\bigr)\Bigr]=\exp\Bigl[-\exp\bigl(\log n - N/n +O(Nn^{-2})\bigr)\Bigr]\\
&\qquad\quad\le \exp\Bigl[-\exp\bigl(\eps\log n+O(n^{-1}\log n)\bigr)\Bigr]=\exp\bigl(-n^{\eps+O(n^{-1})}\bigr).
\end{align*}

By the discussion above, and the union bound, we have
\begin{equation}\label{1.1}
\begin{aligned}
\Bbb P(A_n)&\le \exp\bigl(-(n-1)^{\eps+O(n^{-1})}\bigr) \\
&+\sum_{\nu\ge N_{n-1}}\!\!(1-n^{-1})^{\nu}\,\Bbb E\Biggl[\,\prod_{i\in [n-1]}\Bbb I\bigl(D_{\nu,i}\ge 1\bigr)\Bigl(1-\frac{1}{D_{\nu,i}+1}\Bigr)\Biggr].
\end{aligned}
\end{equation}
{\it Explanation.\/} {\bf (a)\/} $(1-n^{-1})^{\nu}$ is the probability that each of $\nu$ consecutively chosen agents is from $[n-1]$.
$D_{\nu,i}$ is the  number of {\it all\/} proposals received by agent $i$ in the $\nu$-long sequence of proposals, so that
$\sum_{i\in [n-1]}D_{\nu,i}=\nu$, and $1-(D_{\nu,i}+1)^{-1}$
is an upper bound for the conditional probability that agent $i$  prefers the best of the $D_{\nu,i}$ proposers to agent $n$. It is indeed only an upper bound since some of the $D_{\nu,i}$ agents may well be ``repeat-proposers''.

Now, the function $f(z)=1-(z+1)^{-1}$ is log-concave on $(0,\infty)$. Therefore
\begin{multline*}
\prod_{i\in [n-1]}\Bbb I\bigl(D_{\nu,i}\ge 1\bigr)\Bigl(1-\frac{1}{D_{\nu,i}+1}\Bigr)
\le \Biggl(1-\frac{1}{\frac{1}{n-1}\sum_{i\in [n-1]}\bigl(D_{\nu,i}+1\bigr)}\Biggr)^{n-1}\\
=\Biggl(1-\frac{n-1}{\nu+n-1}\Biggr)^{n-1}\le \exp\Biggl(-\frac{(n-1)^2}{\nu+n-1}\Biggr).
\end{multline*}
So the bound \eqref{1.1} becomes
\[
\Bbb P(A_n)\le \exp\bigl(-(n-1)^{\eps+O(n^{-1})}\bigr) +\sum_{\nu\ge 0}\,\, \exp\Biggl(-\frac{\nu}{n}-\frac{(n-1)^2}{\nu+n-1}\Biggr).
\]
Splitting the sum into two sub-sums, for $\nu\le \lfloor n^{3/2}\rfloor$ and $\nu> \lfloor n^{3/2}\rfloor$ respectively, we obtain
that both sub-sums are of order $O\bigl( e^{-n^{1/2-o(1)}}\bigr)$. So, picking $\eps=1/2$, we obtain that $\Bbb P(A_n)\le e^{-n^{1/2-o(1)}}$.
\end{proof}

\subsection{Integral formulas for probabilities/expectations.} Here we will use the $\{X_{i,j}\}$-induced preference system to derive
the integral formulas for the leading probabilities.

Given a fixed-point-free permutation $\Pi$, let $\mathcal C_2=\mathcal C_2(\Pi)$ and $\mathcal C_3=\mathcal C_3(\Pi)$ denote the set of agents from the cycles of $\Pi$ with length $2$, and with length $3$ or more, respectively. Introduce:  
{\bf (a)\/} $E_1=E_1(\Pi)$
the set of all ordered pairs $\{i,j\}$ with $i,\,j\in \mathcal C_2 $, ($i\neq j,\,\Pi(j))$, and $E_1^*=\bigl\{\{i,j\}\in E_1: i<\Pi(j)\bigr\}$; 
$E_1^*$ is a maximal subset of $E_1$ which--for every $\{i,j\}\in E_1$--contains exactly one of $\{i,j\}$ and $\{\Pi(j),\Pi(i)\}(\in E_1$ as well).
 {\bf (b)\/} $E_2=E_2(\Pi)$ the set of
all ordered pairs $\{i,j\}$, ($i\neq j,\,\Pi(j))$, with at least one of  $i$, $j$ from $\mathcal C_3$.
\begin{lemma}\label{lem5}
Let $\Pi$ have no fixed point.  Denoting $\bold x=\{x_i\}_{i\in [n]},\, \bold y=\{y_j\}_{j\in\mathcal C_3}$, and $D=\{(\bold x, \bold y): \bold x\in [0,1]^{n}, \bold y\in [0,1]^{|\mathcal C_3|}\}$, 
we have
\begin{equation}\label{1}
\begin{aligned}
&\qquad\Bbb P(\Pi):=\Bbb P(\Pi\text{ is stable})=\int\limits_{(\bold x, \bold y)\in D}\!\!\!\Bbb P(\Pi|\bold x,\bold y)\,d\bold x\,d\bold y,\\
&\Bbb P(\Pi|\bold x,\bold y):=\!\!\!\!\prod_{\{i,j\}\in E_1^*\cup E_2}\!\!\!(1-x_iz_j),
\,\,\,z_j=\left\{\begin{aligned}&y_j,&&j\in \mathcal C_3,\\
&x_{\Pi(j)},&&j\in \mathcal C_2.\end{aligned}\right.
\end{aligned}
\end{equation}
In addition, the admissible $\bold x,\bold y$ satisfy the condition: for each cycle $C$ with $|C|\ge 3$, either $y_j<x_{\Pi(j)}$ $(\forall\,j\in C)$, or $y_j> x_{\Pi(j)}$ $(\forall\,j\in C)$. 
\end{lemma}

\begin{proof} 
By definition, $\Pi$ is stable if and only if:
\si
 For all  pairs  $i\neq j, \Pi(j)$, if $i$ prefers $\Pi(j)$ to $\Pi(i)$ then $\Pi(j)$ prefers $j$ to $i$. In terms of the matrix $\{X_{u,v}\}$, it is equivalent to 
 the condition ``$X_{i,\Pi(j)}<X_{i, \Pi(i)}\Longrightarrow X_{\Pi(j),j}<X_{\Pi(j),i}$''.  Let us call this the $\star$-condition.

Using the $\star$-condition  for an agent $i$ in a cycle $C$ of length
$3$ or more, and $j=\Pi^{-2}(i)$, we see that if $i$, for instance, prefers their predecessor $\Pi^{-1}(i)$ to their successor $\Pi(i)$, then the predecessor $\Pi^{-1}(i)$ also
prefers their own own predecessor to their own successor. Thus in each such cycle $C$ every agent  has exactly the same relative
preference for their predecessor and their successor.

Given $\bold x=\{x_i\}_{i\in [n]}$, $\bold y=\{y_j\}_{j\in \mathcal C_3}$,
introduce the event
\[
A:=\{X_{i,\Pi(i)}=x_i,\,i\in [n]\}\cap\{X_{\Pi(j),j}=y_j,\,j\in\mathcal C_3\}.
\]
From the $\star$-condition and its discussion, it follows that $\Bbb P(\Pi\text{ is stable}\boldsymbol|A)$\linebreak $=0$ if $\Pi$ has a cycle $C$ of length $3$ or more, such that for two agents $j_1,\,j_2\in  C$ we have $y_{j_1}<x_{\Pi(j_1)}$ and $y_{j_2}>x_{\Pi(j_2)}$. Suppose alternatively that for each cycle $C$, 
$(|C|\ge 3)$, either $y_j<x_{\Pi(j)}$, $(\forall\,j\in C)$, or $y_j> x_{\Pi(j)}$, $(\forall\,j\in C)$. Conditioned on the event $A$, the {\it distinct\/} events meeting the $\star$-condition 
are independent, with probabilities $1-x_i z_j$. The events for $\{i,j\}\in E_2$ are all distinct. However, the seemingly different events for $\{i,j\}\in E_1$ and for $\{i',j'\}=\{\Pi(j), \Pi(i)\}$ ($\in E_1$, as well) 
are one and the same, because $\Pi^2(\ell)=\ell$ for every $\ell\in \mathcal C_2$. We get the distinct events for those $\{i,j\}\in E_1$ by restricting them to $E_1^*$, i.e. keeping only the pairs $\{i,j\}$ with $i<\Pi(j)$. The product of all the probabilities is equal to
\[
\prod_{\{i,j\}\in E_1^*}\!\!\!(1-x_iz_j)\cdot\!\!\!\!\prod_{\{i,j\}\in E_2}\!\!\!(1- x_i z_j),
\]
which proves the formula for $\Bbb P(\Pi|\bold x,\bold y)$ in \eqref{1}.
\end{proof}

Next, we define a rank of  successor $\Pi(i)$ (predecessor $\Pi^{-1}(i)$, resp.) as $1$ plus the total number of agents
$j\neq \Pi(i)$ ($j\neq \Pi^{-1}(i)$ resp.) such that agent $i$ prefers $j$ to $\Pi(i)$ (prefers $j$ to $\Pi^{-1}(i)$, resp.). 
We introduce $R_s(\Pi)$ and $R_p(\Pi)$, the total rank of all $n$ successors and the total rank of all $n$ predecessors,
respectively.  Assuming that a permutation $\Pi$ does not have a fixed point, we have
\begin{equation}\label{3}
\begin{aligned}
R_s(\Pi)&=n+\sum_{i\in [n]} \boldsymbol|\{j: X_{i,\Pi(j)}<X_{i,\Pi(i)}\}\boldsymbol|,\\
R_p(\Pi)&=n+\sum_{j\in [n]} \boldsymbol|\{i: X_{\Pi(j),i}<X_{\Pi(j),j}\}\boldsymbol|.
\end{aligned}
\end{equation}
Let us explain the second formula, for instance. As $j$ runs through $[n]$, $\Pi(j)$ runs through $[n]$ as well. Given $j$, $i$
contributes $1$ to $R_p(\Pi)$ whenever $X_{\Pi(j),i}$ falls below $X_{\Pi(j),j}$, since $j$ is the predecessor of $\Pi(j)$.


\begin{lemma}\label{lem6} Given a permutation $\Pi$ of $[n]$ without a fixed point,  and $k,\ell\ge n$, let $P(k,\ell;\Pi)$ denote the probability that
$\Pi$ is stable, $R_s(\Pi)=k$ and $R_p(\Pi)=\ell)$. Then, denoting $1-u=\overline u$, we have 
\begin{equation}\label{4}
\begin{aligned}
P(k,\ell;\Pi)&=\!\!\!\!\!\int\limits_{(\bold x,\bold y)\in D }\!\!\!\bigl[\xi^{k-n}\eta^{\ell-n}\bigr]
\prod_{\{i,j\}\in E_1^*\cup E_2}\!\!\!\!\!(\overline x_i\overline z_j+ \xi x_i\overline z_j+\eta\overline x_iz_j)\, d\bold xd\bold y;
\end{aligned}
\end{equation}
here $\bold x, \bold y$ also meet the condition from Lemma \ref{lem5}, and the integrand equals the coefficient by $\xi^{k-n}\eta^{\ell-n}$ in the product of the $(\xi,\eta)$-linear polynomials.
\end{lemma}
\begin{proof}
First, introducing $\Bbb I(\Pi):=\Bbb I(\Pi\text{ is stable})$, we have: for $k\ge n$, $\ell\ge n$,
\begin{align*}
\Bbb P(\Pi\text{ is stable}, R_s(\Pi)=k,\,R_p(\Pi)&=\ell\boldsymbol|\bold x,\,\bold y)\\
&=
\bigl[\xi^k\eta^{\ell}\bigr]\,\Bbb E\bigl[\Bbb I(\Pi)\,\xi^{R_s(\Pi)}\,\eta^{R_p(\pi)}\boldsymbol|\bold x,\,\bold y\bigr].
\end{align*}
To evaluate the conditional expectation, it is convenient to treat the attendant generating function probabilistically.  Let
$\xi,\eta\in [0,1]$ be chosen. We sift through the ordered pairs of distinct $i,j$ in any order. Whenever $X_{i,\Pi(j)}<X_{i,\Pi(i)} (=x_i)$,
we ``mark'' $\{i,j\}$ with probability $\xi$; whenever $X_{\Pi(j),i}<X_{\Pi(j),j}(=z_j)$, we ``color'' $\{i,j\}$ with probability $\eta$. Assume that the coloring-marking operations for distinct pairs $\{i,j\}$ are done independently. For completeness only, assume also that if $X_{i,\Pi(j)}<X_{i,\Pi(i)}$ {\it and\/} $X_{\Pi(j),i}<X_{\Pi(j),j}$ then marking and coloring of $\{i,j\}$ are done independently. [No such pair exists on
event ``$\Pi$ is stable''.] Then 
\[
\Bbb E\bigl[\Bbb I(\Pi)\,\xi^{R_s(\Pi)}\,\eta^{R_p(\Pi)}\boldsymbol|\bold x,\,\bold y\bigr]=\xi^n \eta^n \Bbb P(B|\bold x,\bold y),
\]
where $B$ is the event ``$\Pi$ is stable, and all pairs $\{i,j\}$ which are eligible for marking or coloring are marked or colored''. Therefore
$B=\cap_{\{i,j\}}\mathcal B_{i,j}$, where 
\begin{multline*}
\mathcal B_{i,j}=\bigl\{(X_{i,\Pi(i)}<X_{i,\Pi(j)},\,X_{\Pi(j),j}< X_{\Pi(j),i})\\
\text{ or }(X_{i,\Pi(i)}>X_{i,\Pi(j)},\,X_{\Pi(j),j}< X_{\Pi(j),i}\text{ and } \{i,j\}\text{ is 
marked})\\
\text{ or }(X_{i,\Pi(i)}<X_{i,\Pi(j)},\,X_{\Pi(j),j}> X_{\Pi(j),i}\text{ and } \{i,j\}\text{ is colored})\bigr\};
\end{multline*}
notice that the event ``$X_{i,\Pi(j)}<X_{i,\Pi(i)}$ {\it and\/} $X_{\Pi(j),i}<X_{\Pi(j),j}$'' is excluded, as it should be. For $\{i,j\}\in E_1^*\cup E_2$,
the events $\mathcal B_{i,j}$ are conditionally independent, and
\[
\Bbb P(\mathcal B_{i,j}|\bold x,\bold y)=
\overline x_i\overline z_j+\xi x_i\overline z_j+\eta\overline x_iz_j.
\]
Therefore, for $\xi,\eta\in [0,1]$ and thus for all $\xi,\eta$, we have
\begin{align*}
\Bbb E\bigl[\Bbb I(\Pi)\,\xi^{R_s(\Pi)}\,\eta^{R_p(\pi)}\boldsymbol|\bold x,\,\bold y\bigr]&=\xi^n \eta^n\!\!\!\!\!\! \prod_{\{i,j\}\in E_1^*\cup E_2}\!\!\!\!(\overline x_i\overline z_j+ \xi x_i\overline z_j+\eta\overline x_iz_j),
\end{align*}
which proves \eqref{4}.
\end{proof}
{\bf Note.\/}  The integral formulas \eqref{1.5} and \eqref{4} resemble their counterparts for the bipartite stable {\it matchings\/}, see Lemma 3.1 in \cite{Pit2}.   \\

The equations \eqref{4} will be used to determine a likely magnitude of the total number of proposals. 
\section{Likely stable permutations.} 
\noindent {\bf Definition.\/} A stable $\Pi$ is $\Pi_0$-like if each $i$ prefers $\Pi(i)$ to $\Pi^{-1}(i)$. 
\subsection{A lower bound for the likely number of fixed pairs.}
A pair of agents $\{i,j\}$ is called fixed if they form a matched pair in every stable permutation $\Pi$. By Corollary \ref{1}, a pair is fixed if and only if it is a match in $\Pi_0$. 
\begin{theorem}\label{thm1} Let $\mathcal L(\Pi)$ stand for the total (even) cardinality of the matched pairs in a stable $\Pi$.
Let $S_{n;l}$ denote the total number of fixed-point-free, $\Pi_0$-like permutations $\Pi$ with $\mathcal L(\Pi)=l$. 
If $\Delta\in (3/4,1)$, then $\sum_{l\le  n-2n^{\Delta}}\Bbb E[S_{n;l}] = O(n^{-1})$.
Consequently, with probability $1-O(n^{-1})$, the double number of matched pairs in $\Pi_0$ exceeds $n-2n^{\Delta}$ .
\end{theorem}
\subsubsection{Bounding $\Bbb P(\Pi|\bold x)$.} 

The first step in the proof of Theorem \ref{thm1} is 

\begin{proposition}\label{7.5} Let $s:=\sum_{i\in [n]} x_i$, $s_1:=\sum_{i\in\mathcal C_2} x_i$. For all $\bold x\in [0,1]^n$, we have
\begin{equation}\label{2.01}
\Bbb P(\Pi|\bold x)\le
c\,\exp\Bigl(-s^2 +\frac{s_1^2}{2}\Bigr)\Biggl(\!\frac{1-e^{-s}}{s}\!\Biggr)^{m}, \quad c=e^{20}.
\end{equation}
\end{proposition}
\noindent  We will use the bounds
\begin{equation}\label{basic}
\begin{aligned}
&\log(1-u) \le -u - \frac{u^2}{2},\quad u\in [0,1),\\
&\log(1-u) = -u - \frac{u^2}{2}+O(|u|^3),\quad  u\to 0.
\end{aligned}
\end{equation}
\begin{proof}  By Lemma \ref{lem5}, 
and \eqref{1}, we have 
\begin{equation}\label{2.1}
\begin{aligned}
\Bbb P(\Pi|\bold x,\bold y)&:=\Bbb P(\Pi\text{ is stable}|A)=\!\!\!\prod_{\{i,j\}\in E_1^*\cup E_2}\!\!\!\!\!(1-x_iz_j),\\
\end{aligned}
\end{equation}
where $z_j=x_{\Pi(j)}$ for $j\in \mathcal C_2$, $z_j=y_j>x_{\Pi(j)}$ for $j\in \mathcal C_3$.
\begin{lemma}\label{lem6} Denoting  $a(I):=\sum_{i\in I} a_i$, we have
\begin{equation}\label{2.2}
\begin{aligned}
&\,\,\,\,\,\prod_{(i,j)\in E_1^*}\!(1-x_iz_j)\le e^{4.5}\exp\Bigl(-\frac{1}{2}x^2(\mathcal C_2)\Bigr),\\
&\quad\prod_{\{i,j\}\in E_2}\!(1-x_iz_j)\le e^{15}\exp\Bigl(-x(\mathcal C_2)x(\mathcal C_3) -y(\mathcal C_3)\bigl(x(\mathcal C_2)+x(\mathcal C_3)\bigr)\Bigr).
\end{aligned}
\end{equation}
\end{lemma}
\begin{proof} The first inequality is a close version of an inequality
proved in \cite{Pit5}.  Turn to the second inequality.  Let us bound the generic product $\prod_{i\in I, j\in J} (1-x_iz_j)$; here
$\Pi(I)=I$, $\Pi(J)=J$, and  $j\neq i, \Pi^{-1}(i)$. Using the inequality in \eqref{basic} we have
\begin{multline*}
\log\!\!\prod_{i\in I,\,j\in J}\!\!\!(1-x_ iz_j)\le-\sum_{\i\in I,j\in J}x_ iz_j-\frac{1}{2}\sum_{i\in I, j\in J}x_ i^2z_j^2\\
\le -\sum_{i\in I} x_i \Biggl(\sum_{j\in J}z_j-\Bigl(z_{\Pi^{-1}(i)}+z_i)\Bigr)\Bbb I(i\in J)\Biggr)\\
-\frac{1}{2}\sum_{i\in I} x_i^2 \Biggl(\sum_{j\in J}z_j^2-\Bigl(z_{\Pi^{-1}(i)}^2+z_i^2\Bigr)\Bbb I(i\in J)\Biggr)\\
\le -x(I) z(J)-
\frac{1}{2}\Biggl(\sum_{i\in I} x_i^2\Biggr) \Biggl(\sum_{j\in J}z_j^2\Biggr)\\
+\frac{3}{2}\sum_{i\in I} x_i\Bigl(z_{\Pi^{-1}(i)}+z_i\Bigr)\Bbb I(i\in J).
\end{multline*}
Since $\Pi$ is a bijection on $I$, by Cauchy-Schwartz inequality,  the bottom expression is below $3Z^{1/2}$,
where 
\[
Z:=\Biggl(\sum_{i\in I} x_i^2\Biggr) \Biggl(\sum_{j\in J}z_j^2\Biggr).
\]
Therefore
\[
\log\!\!\prod_{i\in I, j\in J}\!\!\!(1-x_ iz_j)\le- x(I) z(J) -0.5 Z+3 Z^{1/2},
\]
and we notice that $-0.5 Z+3 Z^{1/2}=-0.5(Z^{1/2}-3)^2 +4.5<5$. We conclude that
\begin{equation}\label{!}
\log\!\!\prod_{i\in I,\, j\in J}\!\!\!(1-x_ iz_j)\le -x(I) z(J)+5.
\end{equation}
Applying \eqref{!} to $(I=\mathcal C_2,\,J=\mathcal C_3)$, $(I=\mathcal C_3,\,J=\mathcal C_2)$, and $(I=J=\mathcal C_3)$, adding the bounds,
exponentiating the result, and using the definition of $z_j$, we obtain the bottom bound in \eqref{2.2}.
\end{proof}
{\bf Note.\/} The inequality \eqref{!} also holds for all $I$ and $J$ and $n$ sufficiently large,
if $\max_{i\in [n]} x_i\le \delta_n\to 0$ and $\bold z\in \Bbb R^n$ is such that $\max_{j\in [n]} |z_j|\le a$, $a$ being fixed.
The proof is based on the second estimate in \eqref{basic}, but otherwise it is a minor variation of the argument above.

\begin{corollary}\label{cor2} Denoting  $c=e^{20}$, we have
\begin{align*}
\Bbb P(\Pi|\bold x,\bold y)&\le c\exp\Bigl(-\frac{1}{2}x^2(\mathcal C_2)-x(\mathcal C_2)x(\mathcal C_3)\Bigr)\times e^{-sy(\mathcal C_3)}. 
\end{align*}
\end{corollary}
\noindent Hence
\begin{equation}\label{6.69}
\begin{aligned}
\Bbb P(\Pi|\bold x)&=\int\limits_{ y_j> x_{\Pi(j)})}\!\!\!\!\!\!\Bbb P(\Pi|\bold x,\bold y)\,d\bold y\\
&\le 
c\,\exp\Bigl(-\frac{1}{2}x^2(\mathcal C_2)-x(\mathcal C_2)x(\mathcal C_3)\Bigr) \prod_{j\in \mathcal C_3}\int_{x_j}^1 e^{-s y}\,dy.
\end{aligned}
\end{equation}
The function 
\begin{equation}\label{6.7}
f(x):=\int_x^1e^{-sy}\,dy=s^{-1}(e^{-sx}-e^{-s})
\end{equation}
is log-concave since $(\log f(x))^{''}=-f^{-2}(x)e^{-s(x+1)}<0$. So, setting $x_{\text{ave}}(\mathcal C_3)=m^{-1}x(\mathcal C_3)$,
$(m:=|\mathcal C_3|\le n-2)$, we have
\begin{equation}\label{6.71}
\begin{aligned}
\prod_{j\in \mathcal C_3}\int_{x_j}^1 \!\!\!e^{-sy} dy&\le f(x_{\text{ave}}(\mathcal C_3) )^{m}
=e^{-sx(\mathcal C_3)} \!\Biggl(\!\frac{1-e^{-s(1-x_{\text{ave}}(\mathcal C_3))}}{s}\!\Biggr)^{m}\\
&\le e^{-sx(\mathcal C_3)}\!\Biggl(\!\frac{1-e^{-s}}{s}\!\Biggr)^{m}.
\end{aligned}
\end{equation}
Combining \eqref{6.69} and \eqref{6.71} we have
\begin{equation}\label{6.711}
\Bbb P(\Pi|\bold x)\le
c\,\exp\Bigl(-s^2 +\frac{1}{2}x^2(\mathcal C_2)\Bigr)\Biggl(\!\frac{1-e^{-s}}{s}\!\Biggr)^{m}.
\end{equation}
Here $x(\mathcal C_2)=\sum_{i\in \mathcal C_2}x_i=s_1$. The proof of Proposition \ref{7.5} is complete.
\end{proof}
\subsubsection{Preliminaries}To obtain a good bound for $\Bbb P(\Pi)$, we plan to integrate, with sufficient accuracy, the bound \eqref{6.711} of $\Bbb P(\Pi|\bold x)$ against  the uniform probability density on $[0,1]^n$. For this purpose, we will need the following result, cf.  \cite{Pit1}, \cite{Pit2.5} and \cite{Pit5}.
\begin{lemma}\label{lem10} Let $X_1,\dots,X_{n}$ be independent $[0,1]$-Uniforms. Let $S=\sum_{i\in [n]}X_i$, 
and $\bold V=\bigl\{V_i=X_i/S: i\in [n]\bigr\}$. Let $\bold L=\{L_i: i\in [n]\}$ be the set of lengths of the $n$ consecutive subintervals of $[0,1]$ obtained
by selecting, independently and uniformly at random, $n-1$ points in $[0,1]$ Then the joint density $g(s,\bold v)$ of $(S, \bold V)$ is given by
\begin{equation}\label{6.712}
\begin{aligned}
g(s,\bold v)&:=s^{n-1}\,\Bbb I\Bigl(\max_{i\in [n]}v_i\le s^{-1}\Bigr)\,\Bbb I(v_1+\cdots+v_{n-1}\le 1)\\
&\le \frac{s^{n-1}}{(n-1)!}\,g(\bold v);
\end{aligned}
\end{equation}
here $v_{n}:=1-\sum_{i\in [n-1]}v_i$, and $g(\bold v)=(n-1)!\,\Bbb I(v_1+\cdots+v_{n-1}\le 1)$ is the joint density of  $\bold L$.
\end{lemma}
The identity/bound \eqref{6.712} is instrumental in combination with 
\begin{lemma}\label{lem11} {\bf (1)\/} Let $l\le n$, $t\ge 1$ and $\delta<1/2$. There exists $a=a(\delta)>0$ such that %
\begin{equation}\label{6.713}
\Bbb P\Biggl(\Biggr|\frac{n}{l}\sum_{j\in [l]}L_j\Biggr|\ge l^{-\delta}\!\Biggr)\!\le\!e^{-a\,l^{1-2\delta}}. 
\end{equation}
{\bf (2)\/} Introduce $h(\eps)=\log(1-\eps)+\eps;\,\, h(\eps)$ is negative on $(-\infty,1)$. If $m:=n-l=o(n)\to\infty$ and $\eps\in (0,1)$ is fixed, then
\begin{equation}\label{6.714}
\Bbb P\biggl(1-\Bigl(\sum_{i\in [l]}L_i\Bigr)^2\ge \frac{2(1-\eps)}{1+\eps}\bigl(m/n -0.5(m/n)^2\bigr)\biggr)\ge 1- 1.1 e^{mh(\eps)}.
\end{equation}
{\bf (3)\/} 
\begin{equation}\label{6.7145}
\Bbb P\Bigl(\max_{i\in [n]}L_i\ge 2.02\frac{\log n}{n}\Bigr)\le e^{-2.01\log n}.
\end{equation}
\end{lemma}

{\bf Notes.\/} The bounds \eqref{6.713}--\eqref{6.7145} are versions of the bounds proved in \cite{Pit5}.  An easy proof of \eqref{6.7145}
is based on the fact that each of $L_1,\dots, L_n$ is distributed as $\min\{X_1,\dots, X_{n-1}\}$. The key
element of the argument for \eqref{6.713} and \eqref{6.714} is a known result, Karlin and Taylor \cite{KarTay}: $\{L_i\}_{i\in [n]}$
has the same distribution as $\{\tfrac{w_j}{W_n}\}_{j\in [n]}$, where $w_1,\dots, w_{n}$ are independent, $\Bbb P(w_j>x)=e^{-x}$,
$W_{n}:=\sum_{j\in [n]}w_j$. For completeness, we put the proof of the bound \eqref{6.714} into Appendix.

\subsubsection{Continuing the proof of Theorem \ref{thm1}.} Recall the notations $l=|\mathcal C_2|$, $m=|\mathcal C_3|$, i.e. $l+m=n$; $s=\sum_{i\in [n]}x_i$, $v_i:=\frac{x_i}{s}$. Given $D\subset [0,1]^n$, we will denote by $S_D=S_D(\{X_{i,j}\})$ ($S_{D;\,l}$ resp.) the total number of $\Pi$'s such that $\{X_{i,\Pi(i)}\}\in D$ ($\{X_{i,\Pi(i)}\}\in D$ and $|\mathcal C_2|=l$, resp.). 
{\bf (1)\/} 
Introduce 
\[
D_1=\Bigl\{\bold x\in [0,1]^n: \max_{i\in [n]} v_i\le 2.02\frac{\log n}{n}\Bigr\}.
\]
Observe that, by \eqref{6.7145}, 
\begin{equation}\label{6.7146}
\Bbb P\Bigl(\max_{i\in [n]}L_i\ge 2.02\frac{\log n}{n}\Bigr)\le e^{-2.01\log n}.
\end{equation}
Using the notation $\Bbb P_D(\Pi)=\int_{\bold x\in D}P(\Pi|\bold x)\,d\bold x$, we have: by \eqref{6.711}, \eqref{6.712} and \eqref{6.7146},
\begin{equation}\label{6.71465}
\Bbb P_{D_1^c}(\Pi)\le ce^{-2.01\log n}\int\limits_0^\infty e^{-s^2/2}\Biggl(\!\frac{1-e^{-s}}{s}\!\Biggr)^{m}\frac{s^{n-1}}{(n-1)!}\,ds.
\end{equation}
For the integral (call it $I_1$) we have
\[
I_1\le \int\limits_0^\infty e^{-s^2/2}\,\frac{s^{\max(l-1,0)}}{(n-1)!}\,ds=O\Bigl(\frac{(l-2)!!}{(n-1)!}\Bigr), 
\]
where
\[
\nu!!=\left\{\begin{aligned}
&\prod_{\mu\in [\nu]}\mu\cdot \Bbb I\bigl(\mu=\nu (\text{mod } 0)\bigr),&& \nu>0,\\
&\,\,\,\,1,&& \nu\le 0.\end{aligned}\right.
\]
Therefore 
\begin{equation}\label{6.7147}
\Bbb P_{D_1^c}(\Pi)\le O\biggl(\frac{e^{-2.01\log n}(l-2)!!}{(n-1)!}\biggr).
\end{equation}
The total number of fixed-point-free permutations with $m=n-l$ agents in cycles of length $3$ or more is at most
$\binom{n}{m} m! (l-1)!!$. So, by \eqref{6.7147}, we have
\begin{equation}\label{6.7148}
\begin{aligned}
\Bbb E[S_{D_1^c;\,,l}]&\!=\!O\biggl(\!e^{-2.01\log n}\,\frac{\binom{n}{m}\, m!\, (l-1)!!\, (l-2)!!}{(n-1)!}\biggr)\\
&=O\Bigl(\frac{n}{l}n\,e^{-2.01\log n}\Bigr), 
\end{aligned}
\end{equation}
implying that 
\[
\Bbb E[S_{D_1^c}]=\sum_l\Bbb E[S_{D_1^c;\,l}] =O\bigl(n(\log n)e^{-2.01\log n}\bigr)=O(n^{-1}).
\]

{\bf (2)\/}  From now we consider only $\Pi$'s counted in $S_{D_1}$.
Suppose that 
\begin{equation}\label{6.71485}
l\le \frac{0.5\log^2 n}{\log\log n}.
\end{equation}
 For $\bold x\in D_1$, we have
\[
s_1:=x(\mathcal C_2)=s\sum_{i\in \mathcal C_2}v_i =O(s\eps),\quad \eps=\frac{ l \log n}{n}.
\]
For $\Pi$ counted in $S_{n;l}$, by \eqref{6.711} and \eqref{6.712}, we bound
\[
\Bbb P_{D_1}(\Pi)\le cI_2,\quad I_2:=\int\limits_{0}^{\infty}\exp\Bigl(-s^2(1-O(\eps^2))\Bigr)\biggl(\!\frac{1-e^{-s}}{s}\!\biggr)^{m}\frac{s^{n-1}}{(n-1)!}\,ds.
\]
Substitute $\eta=s(1-O(\eps^2))^{1/2}$. Since $\frac{1-e^{-z}}{z}$ decreases as $z$ increases, we obtain
\[
I_2\le (1-O(\eps^2))^{-l/2} \hat I_2,\quad \hat I_2:=\int\limits_{0}^{\infty}e^{-\eta^2}\biggl(\!\frac{1-e^{-\eta}}{\eta}\!\biggr)^{m}\frac{\eta^{n-1}}{(n-1)!}\,d\eta.
\]
The factor $(1-O(\eps^2))^{-l/2}$ is 
\[
1+O\bigl(l\eps^2 \bigr)=1+O(n^{-2}\log^8 n)\to 1.
\]
The integrand for $\hat I_2$ is at most $e^{H_{m,n}(\eta)}$, where
\begin{equation}\label{6.7149}
H_{m,n}(\eta):=-\eta^2+m\log(1-e^{-\eta}) +(n-m)\log\eta.
\end{equation}
Since
\[
H^{''}_{m,n}(\eta)=-2-\frac{m e^{\eta}}{(e^{\eta}-1)^2}-\frac{n-m}{\eta^2}<0,
\]
$H_{m,n}(\eta)$ is concave. So a stationary point of $H_{m,n}(\eta)$, if it exists, is the maximum point. A stationary point is the root of 
\[
H'_{m,n}(\eta)=-2\eta+\frac{m}{e^{\eta}-1}+\frac{n-m}{\eta}=0.
\]
Pick $\eta(\be)=\log\frac{\be m}{\log m}$. Then
\[
H_{m,n}'(\eta(\be))=-(2-\be^{-1})\log m +O\biggl(\frac{\log m}{\log\log m}\biggr).
\]
It shows that, for $n$ sufficiently large, $H_{m,n}(\eta)$ does have a stationary (whence maximum) point $\eta^*=\log\frac{\be^* m}{\log m}$, ($\be^*=\be(m,n)\in [0.49, 0.51]$, say). Therefore, by \eqref{6.7149},
\begin{align*}
H_{m,n}(\eta^*)&=-(1-o(1))\log^2m+(n-m)\log\eta^*\\
&\le -(1-o(1))\log^2m+ (1+o(1))0.5 \log^2m\\
&=-(0.5 -o(1))\log^2 n.
\end{align*}
Finally, $H^{''}_{m,n}(\eta)\le -2$, which easily leads to a bound
\[
I_3=O\bigl(\exp(H_{m,n}(\eta^*))\bigr)=O\bigl(\exp(-0.4\log^2n)\bigr).
\]
Analogously to $\Bbb E[S_{D_1^c;\,l}]$ in \eqref{6.7148}, we conclude: under the condition \eqref{6.71485},
\begin{equation}\label{6.7151}
\Bbb E[S_{D_1;\,l}]= O\bigl(\exp(-0.3\log^2 n)\bigr),\quad l\le\frac{0.5\log^2 n}{\log\log n}.
\end{equation}

Consider now a complementary case 
\begin{equation}\label{6.7152}
\ell\in \Bigl[\frac{0.5\log^2n}{\log\log n},\,n- 2 n^{1-\Delta}\Bigr].
\end{equation}
Introduce
\begin{equation}\label{6.7153}
D_2:=\Bigl\{\bold x\in D_1: \sum_{i\in \mathcal C_2} v_i\le \frac{\ell}{n}(1+\ell^{-\Delta})\Bigr\}.
\end{equation}
Observe that by \eqref{6.713}
\[
\Bbb P\Bigl(\sum_{i\in \mathcal C_2}L_i\ge \frac{\ell}{n}(1+\ell^{-\Delta})\Bigr)\le e^{-a\ell^{1-2\Delta}},\quad (\Delta<1/2).
\]
Like $\Bbb E[S_{D_1^c;\,l}]$ in \eqref{6.7148}, we have 
\begin{align*}
\Bbb E[S_{D_1\setminus D_2;\,l}]&=O\Bigl(n e^{-a\ell^{1-2\Delta}}\log n\Bigr)\\
&=O\biggl[\exp\Bigl(\log n-a\Bigl(\frac{\log^2 n}{\log\log n}\Bigr)^{1-2\Delta}\Bigr)\log n\biggr].
\end{align*}
If $\Delta<1/4$, which we assume from now, then this bound yields that 
\begin{equation}\label{6.7154}
\Bbb E[S_{D_1\setminus D_2;\,l}]\le e^{-\log^{1+\sigma}\!n},\quad \forall\,\sigma<1-4\Delta.
\end{equation}
Turn to $\Bbb E[S_{D_2;l}]$. Using \eqref{6.711}, \eqref{6.712} and the definition of $D_2$ in \eqref{6.7153}, we have:
for $\Pi$ counted in $S_{n;l}$, 
\begin{align*}
\Bbb P_{D_2}(\Pi)&\le c  \int\limits_0^{\infty}\exp\biggl(-s^2\Bigl(1-(l/n)^2(1+l^{-\Delta})^2/2\Bigr)\biggr) \frac{s^{l}}{(n-1)!}\,ds\\
&=O\Biggl(\frac{(l-2)!!}{(n-1)! \Bigl(2-(l/n)^2(1+l^{-\Delta})^2\Bigr)^{l/2}}\Biggr).
\end{align*}
To bound the $l$-dependent factor in the denominator, we use an inequality $\frac{l}{n}\le 1-l^{-\Delta}$, true for $l$ meeting the
condition  \eqref{6.7152}. (Indeed, $\phi(x):=\frac{x}{n}+x^{-\Delta}-1$ is convex; so it is negative for admissible $x$ if it is negative at both
ends of the interval in \eqref{6.7152}, which is easy to check.) By this inequality, we have
\[
\Bigl(2-(l/n)^2(1+l^{-\Delta})^2\Bigr)^{l/2}\ge \Bigl(1+\ell^{-2\Delta}\Bigr)^{\ell/2}\ge \exp\bigl( 0.4 \ell^{1-2\Delta}\bigr),
\]
whence, $\forall\,\sigma<1-4\Delta$,
\[
\Bbb P_{D_2}(\Pi)=O\biggl(\frac{(\ell-2)!!}{(n-1)! \exp\bigl( 0.4 \ell^{1-2\Delta}\bigr)}\biggr)\le \frac{(\ell-2)!!}{(n-1)!} e^{-\log^{1+\sigma}n}.
\]
Consequently $\Bbb E[S_{D_2;\,l}]\le e^{-\log^{1+\sigma} n}$, which, combined with \eqref{6.7151}, implies that, for all $l\le n-2n^{1-\Delta}$,
\begin{equation}\label{6.7155}
\Bbb E[S_{D_1;\,l}]=\Bbb E[S_{D_1\setminus D_2;\,l}]+\Bbb E[S_{D_2;l}]\le e^{-\log^{1+\sigma} n},\quad\forall\,\sigma<1-4\Delta.
\end{equation}
Therefore, for every $\Delta<1/4$, we have
\[
\sum_{l\le n-2n^{1-\Delta}}\Bbb E[S_{D_1;\,l}]\le e^{-0.9\log^{1+\sigma} n},\quad \forall\,\sigma<1-4\Delta.
\]
The proof of Theorem \ref{thm1} is complete.\\

\subsection{On the expected number of fixed pairs and the expected number of stable permutations.} 
So we have proved that the total cardinality of cycles with length $3$ or more in any $\Pi_0$-like permutation $\Pi$ is likely to be $n^{3/4+o(1)}$ at most. This bound can be
significantly improved.
\begin{theorem}\label{thm2} Let $S_n$ ($S_n^+$ resp.) denote the total number of $\Pi_0$-like permutations (the total number
of all stable permutations, resp.) and $\mathcal L_n:=\mathcal L(\Pi_0)$.
 Both $\Bbb E[S_n]$ and $\Bbb E[n-\mathcal L_n]$ are bounded as $n\to\infty$. Consequently $S_n^+$ is bounded in probability, i.e.
 $\Bbb P(S_N^+\le \omega(n))\to 0$, if $\omega(n)\to\infty$ however slowly.
\end{theorem}
\begin{proof} It suffices to consider $\Pi$'s with $l=l(\Pi)\ge n-2n^{1-\Delta}$, $(\Delta<1/4)$. By \eqref{6.711}, we have: with  $s_1=\sum_{i\in [l]} x_i$,
\begin{equation}\label{6.7156}
\begin{aligned}
\Bbb P(\Pi|\bold x)&\le c s^{-m}\exp\Bigl(-s^2+\frac{1}{2}s_1^2\Bigr)\\
&=c s^{-m}\exp\biggl(-s^2\biggl(1-\frac{1}{2}\Bigl(\sum_{i\in [l]} v_i\Bigr)^2\biggr)\biggr),\quad (m=n-l).
\end{aligned}
\end{equation}
And, of course, $\Bbb P(\Pi|\bold x)\le c s^{-m} e^{-s^2/2}$. We need
 
The equations \eqref{6.711} and \eqref{6.7156}, combined with Lemma \ref{lem11},  yield
\begin{align*}
\Bbb P(\Pi)&\le \frac{c}{(n-1)!}\int\limits_{s, \bold v} s^{l-1} \exp\biggl(-s^2\biggl(1-\frac{1}{2}\Bigl(\sum_{i\in [l]} v_i\Bigr)^2\biggr)\biggr) g(\bold v)\,
ds \prod_{i\in [n-1]} dv_i\\
&\le \frac{c}{(n-1)!}\int_0^{\infty}s^{l-1}\exp\biggl[-\frac{s^2}{2}\biggl(1+\frac{2(1-\eps)}{1+\eps}\bigl(m/n -0.5(m/n)^2\bigr)\biggr)\biggr]\,ds\\
&\qquad+\frac{c}{(n-1)!} \int_0^{\infty} s^{l -1}e^{-s^2/2}\,ds\, \cdot\, O\bigl(e^{mh(\eps)}\bigr)\\
&\le\frac{c_1(l-2)!!}{(n-1)!}\cdot\biggl[\biggl(1+\frac{2(1-\eps)}{1+\eps}\bigl(m/n -0.5(m/n)^2\bigr)\biggr)^{-\frac{l}{2}}+
e^{mh(\eps)}\biggr]\\
&\le \frac{c_2 (l-2)!!}{(n-1)!} e^{(1+o(1)) m f(\eps)},\quad f(\eps):=\max\Bigl(-\frac{1-\eps}{1+\eps};\, h(\eps)\Bigr)<0.
\end{align*}
For the last step we used the bound \eqref{basic} and $l\sim n$. ($f(\eps)$ attains its minimum $\approx -0.272$ at $\eps^*\approx 0.573$,
and $e^{f(\eps^*)}\approx 0.763$.)
Let $\omega(n)\to\infty$ however slowly.  Then, using the bound above, we obtain
\begin{multline}\label{6.7157}
\sum_{l= n-2n^{1-\Delta}}^{n -\omega(n)}\,\sum_{\Pi:l(\Pi)=l}\! \Bbb P(\Pi)
\le \sum_{l= n-2n^{1-\Delta}}^{n -\omega(n)}\!\!\frac{\binom{n}{m} m! (l-1)!}{(n-1)!}e^{(f(\eps^*)+o(1)) m }\\
=\sum_{l= n-2n^{1-\Delta}}^{n -\omega(n)}\! \frac{n}{l}\, (0.77)^{n-l} 
 = (0.78)^{\omega(n)}\to 0.
\end{multline}
It follows from Theorem \ref{1} and \eqref{6.7157} that both $\Bbb E\bigl[n - \mathcal L_n\bigr]$ and $\Bbb E[S_n]$ are at most
\[
o(1) +O\biggl(\sum_{l=n-\omega(n)}^n\frac{n}{l}\biggr)=O(\omega(n)),
\]
for $\omega(n)\to\infty$ however slowly. Thus $\Bbb E\bigl[n - \mathcal L_n\bigr]=O(1)$,  $\Bbb E[S_n]=O(1)$.
\end{proof}
\subsubsection{Likely range of $s(\Pi)=\sum_{i\in [n]}X_{i,\Pi(i)}$.} The proof above demonstrated that the contribution of  stable
$\Pi$'s to $\Bbb E[S_n]$ such that $s^{-1}(\Pi)\max_{i\in [n]} X_{i,\Pi(i)}\ge 2.02 n^{-1}\log n$ or $n-\mathcal L(\Pi)\ge \omega(n)$ is vanishingly small as $n\to\infty$. Let us show also that the values of $s(\Pi)$ contributing most to $\Bbb E[S_n]$ are sharply concentrated around $n^{1/2}$. Introduce $D_3=\{\bold x\in [0,1]^n: |s-(l-1)^{1/2}|\ge (2\log n)^{1/2}\}$. Then, like several times earlier, we bound
\[
\Bbb P_{D_3^c}(\Pi)\le c\!\!\!\int\limits_{|s-(l-1)^{1/2}|\ge (2\log n)^{1/2}}\!\!\!\!\!\!\!\!\! e^{-s^2/2}\frac{s^{l-1}}{(n-1)!}\,ds.
\]
The integrand, call it $\Psi(s)$, attains its maximum at $s_{\text{max}}=(l-1)^{1/2}$,
and $(\log \Psi(s))^{''}<-1$ for all $s>0$. Using an inequality
\[
\int_{\eta \ge a} e^{-\eta^2/2}\,d\eta\le a^{-1} e^{-a^2/2},\,\,(a>0),
\]
we obtain
\begin{align*}
\Bbb P_{D_3^c}(\Pi)&\le \frac{c\,\Psi\bigl(s_{\text{max}}\bigr)}{(n-1)!}\int\limits_{|\eta|\ge (2\log n)^{1/2}}\!\!\!\!\!\!\! e^{-\eta^2/2}\,d\eta
\le \frac{c\,\Psi\bigl(s_{\text{max}}\bigr)}{n^2(n-1)!} \\
&\le \frac{c_1}{n^2(n-1)!}\int_{s\ge 0} e^{-s^2} s^{l-1}\,ds=\frac{c_1(l-2)!!}{n(n-1)!}.
\end{align*} 
(For the third inequality we used $(\log\Psi(s))^{''}\ge -5$ for $s\ge 0.5 n^{1/2}$.) So 
\begin{equation}\label{6.7158}
\Bbb E[S_{D_3^c}]=o(1)+\sum_{l\ge n-\omega(n)}\Bbb E[S_{D_3^c;\,l}]=O\bigl(n^{-1}\omega(n)\bigr).
\end{equation}

\subsection{Likely ranks of successors and predecessors.} 

Let $R_s(\Pi)$ and $R_p(\Pi)$ stand for the total rank of successors and the total rank of predecessors in a fixed-point-free stable permutation $\Pi$.
\begin{theorem}\label{thm3} With probability $\ge 1-o(1)$, we have: for all $\Pi_0$-like permutations $\Pi$,
\[
\frac{n^{3/2}}{2}\Bigl(1-2.2 n^{-1/2}\log n\Bigr)\le R_s(\Pi)\le R_p(\Pi) \le \frac{n^{3/2}}{2}\Bigl(1+2.2 n^{-1/2}\log n\Bigr).
\]
In particular, w.h.p. the total number of steps in the proposal algorithm, i.e. $R_s(\Pi_0)$, is asymptotic to $0.5 n^{3/2}$.
\end{theorem}

\begin{proof} {\bf (I)\/} Let $\Bbb P(k;\Pi)$ stand for the probability that $\Pi$ is stable and $R_s(\Pi)=k$.
Setting $\eta=1$  in \eqref{3} (Lemma \ref{lem6}), we obtain
\begin{equation}\label{6.8}
\Bbb P_s(k;\Pi)=\!\!\!\!\!\int\limits_{(\bold x,\bold y)\in D }\!\!\!\bigl[\xi^{k-n}\bigr]
\prod_{\{i,j\}\in E_1^*\cup E_2}\!\!\!\!\!(\overline x_i+ \xi x_i\overline z_j)\, d\bold xd\bold y.
\end{equation}
Introduce $\Bbb P_s(\ge k;\Pi)$ ($\Bbb P_s(\le k;\Pi)$, resp.) the probability that $\Pi$ is stable and $R_s(\Pi) \ge k$  ($R_s(\Pi) \le k$ resp.). From \eqref{6.8} we obtain (Chernoff-type) bounds
\begin{equation}\label{6.81}
\begin{aligned}
\Bbb P_s(\ge k;\Pi)&\le \!\!\!\!\!\int\limits_{(\bold x,\bold y)\in D }\!\!\!\inf_{\xi\ge 1} \biggl[\xi^{n-k}\cdot\!\!\!\!\!\!\!
\prod_{\{i,j\}\in E_1^*\cup E_2}\!\!\!\!\!(\overline x_i+ \xi x_i\overline z_j)\biggr]\, d\bold xd\bold y,\\
\Bbb P_s(\le k;\Pi)&\le \!\!\!\!\!\int\limits_{(\bold x,\bold y)\in D }\!\!\!\inf_{\xi\le 1} \biggl[\xi^{n-k}\cdot\!\!\!\!\!\!\!
\prod_{\{i,j\}\in E_1^*\cup E_2}\!\!\!\!\!(\overline x_i+ \xi x_i\overline z_j)\biggr]\, d\bold xd\bold y.
\end{aligned}
\end{equation}
A prospect of finding even a suboptimal $\xi$ as a function of $\bold x$ and $\bold y$ would have been a real turn-off. Fortunately by now we know that a non-zero limiting contribution to the expected number of $\Pi_0$-like stable permutations  comes exclusively from $(\Pi, \{X_{i,j}\})$'s  meeting the constraints
\begin{equation}\label{6.815}
\begin{aligned}
 &\qquad\quad n-l\le \omega(n),\quad l=\big|\{i\in [n]:\,i\in \mathcal C_2(\Pi)\}\big|,\\
 & |s(\Pi)-(l-1)^{1/2}|\le (2\log n)^{1/2},\quad\max_{i\in [n]} X_{i,\Pi(i)}\le 2.03 n^{-1/2}\log n,
\end{aligned}
 \end{equation}
$\omega(n)\to\infty$ however slowly.
We recall also that our argument used
the exponential bounds for the integrand in Lemma \ref{lem5} and its descendants, and it led  us to an integrand dependent
on $s=\sum_{i\in [n]}x_i$ and $s_1=\sum_{i\in [l]}x_i$ only. It turns out that the same reduction works for the integrands in \eqref{6.81}, if we lower our sights and look for a suboptimal
$\xi$ among functions of $s$ and $s_1$.  So we may consider those $(\Pi, \{X_{i,j}\})$'s only, in which case the range $D$ in \eqref{6.81}
gets replaced with 
\[
D^*:=\Bigl\{\bold x\in D: |s-(l-1)^{1/2}|\le (2\log n)^{1/2};\,\,\,\max_{i\in [n]} x_i\le 2.03 n^{-1/2}\log n\Bigr\}, 
\]
and the LHS's  in \eqref{6.81} become $\Bbb P_s^*(\ge k;\Pi)$ ($\Bbb P_s^*(\le k;\Pi)$,
resp.), the probability that $\Pi$ is stable, $(\Pi, \{X_{i,j}\})$ meets the constraint \eqref{6.815}, and $R_s(\Pi)\ge k$ ($R_s(\Pi)\le k$, resp.). \\

Observe that
\[
\overline x_i+ \xi x_i\overline z_j=1-x_i  z_j (\xi), \quad z_j(\xi):=1-\xi\, \overline z_j=\left\{\begin{aligned}
&1-\xi\,\overline x_j,&&i\in \mathcal C_2,\\
&1-\xi\,\overline y_{\Pi(j)},&&j\in \mathcal C_3.\end{aligned}\right.
\]
Here $\max_i x_i\le 2.03 n^{-1/2}\log n\to 0$, and $\max_j |z_j(\xi)|\le 1+\xi$ is bounded as $n\to\infty$, if we impose the condition $\xi\le 2$, say. So (see the note following the proof of Lemma \ref{lem6}) analogously to Corollary \ref{cor2}, we have
\begin{multline}\label{6.82}
\prod_{\{i,j\}\in E_1^*\cup E_2}\!\!\!\!\!(\overline x_i+ \xi x_i\overline z_j)
\le c \exp\biggl[-\frac{1}{2}\biggl(\sum_{i\in \mathcal C_2} x_i\biggr)\biggl(\sum_{j\in \mathcal C_2}(1-\xi\, \overline x_j)\biggr)\\
-\biggl(\sum_{i\in \mathcal C_3}x_i\biggr)\biggl(\sum_{j\in \mathcal C_2}(1-\xi\,\overline x_j)\biggr) -s\sum_{j\in \mathcal C_3}
(1-\xi\,\overline y_j)\biggr],
\end{multline}
where $s=\sum_{i\in [n]}x_i$. Given $\bold x$, we integrate this inequality for $y_j\ge x_{\Pi(j)}$. Leaving
$\bold x$-dependent factors aside, and recalling the notation $m= |\mathcal C_3|$, we compute
\begin{multline}\label{6.82}
\prod_{j\in \mathcal C_3}\int_{x_{\Pi(j)}}^1 e^{-s(1-\xi +\xi y)}\,dy=e^{-ms(1-\xi)}\prod_{j\in \mathcal C_3}\frac{e^{-s \xi x_{\Pi(j)}}
-e^{-s\xi}}{s\xi}\\
\le e^{-ms(1-\xi)}\biggl(\frac{e^{-s \xi x_{\text{ave}}(\mathcal C_3)}-e^{-s\xi}}{s\xi}\biggr)^m\\
\le \exp\biggl(-ms(1-\xi)-s\xi\sum_{j\in \mathcal C_3}x_j\biggr)\biggl(\frac{1-e^{-s\xi}}{s\xi}\biggr)^m.
\end{multline}
Putting together \eqref{6.81} and \eqref{6.82} and doing a simple algebra we get 
\begin{equation}\label{6.83}
\int\limits_{y_j\ge x_{\Pi(j)}} \prod_{\{i,j\}\in E_1^*\cup E_2}\!\!\!\!\!(\overline x_i+ \xi x_i\overline z_j)\,d\bold y \le c\exp\bigl(H_k(\bold x,\xi)\bigr).
\end{equation}
Here
\begin{multline}\label{6.84}
H_k(\bold x,\xi)=\frac{1}{2}\biggl(\sum_{i\in \mathcal C_2} x_i\biggr)\!\biggl(\sum_{j\in \mathcal C_2}(1-\xi\, \overline x_j)\!\biggr)\\
-s\sum_{i\in [n]}(1-\xi\,\overline x_i)-m\log(s\xi)-(k-n)\log\xi\\
=\frac{1}{2} s_1\bigl(l-\xi(l-s_1)\bigr) -s\bigl(n-\xi(n-s)\bigr)\\
-m\log(s\xi)-(k-n)\log\xi,
\end{multline}
the bottom expression being a function of $s$, $s_1$ and $\xi$ only, call it $h_k(s,s_1,\xi)$. As a function of $\xi$, $h_k$ has a single
stationary point $\xi(s,s_1)$, the root of $(h_k)'_{\xi}(s,s_1,\xi)=0$, given by
\begin{equation*}
\xi_k(s,s_1)=\frac{k+m-n}{s(n-s)-\frac{1}{2} s_1(l-s_1)}.
\end{equation*}
Since $(h_k)^{''}_{\xi}=(k+m-n)/\xi^2>0$, $\xi_k(s,s_1)$ is a unique minimum point of $h_k(s,s_1,\xi)$, given $s$ and $s_1$. Since $n-l$, $s-s_1$ are of order $O(\omega)$, and $|s -( l-1)^{1/2}|\le (2\log n)^{1/2}$, we see that $\xi_k(s,s_1)$ is asymptotic to $2(k+m-n)/ n^{3/2}$. This is a strong indication that the likely values $k$
of $R_s(\Pi)$ are relatively close to $0.5n^{3/2}$, as they make Chernoff-type parameter $\xi_k$ asymptotic to $1$.

Let us make this rigorous. Consider $k\in [k_1,k_2]$, $k_1=\lfloor 0.5 n^{3/2}-n\log n\rfloor$, $k_2=\lfloor 0.5 n^{3/2}+n\log n\rfloor$. It is
easy to check that
\[
\xi_k(s,s_1)=\frac{2k}{n^{3/2}}\bigl(1+O(n^{-1/2}\log^{1/2} n)\bigr),\quad 1-\xi_k(s,s_1)=O\bigl(n^{-1/2}\log n\bigr),
\]
uniformly for $k$ and $l,\,s,\,s_1$ in question. (In particular, $\xi_k(s,s_1)\le 2$, i.e. bounded as $n\to\infty$.) To evaluate $h_k(s,s_1,\xi_k)$, we write
\begin{align*}
h_k(s,s_1,1)&=h_k(s,s_1,\xi_k)+(h_k)'_{\xi}(s,s_1,\xi_k)(1-\xi_k)\\
&+\frac{(h_k)''_\xi(s,s_1,\xi_k)}{2} (1-\xi_k)^2 +O\bigl(k(1-\xi_k)^3\bigr)\\
&=h_k(s,s_1,\xi_k)+\frac{k+m-n}{2\xi_k^2}(1-\xi_k)^2+O\bigl(k(1-\xi_k)^3\bigr)\\
&\ge h_k(s,s_1,\xi_k)+0.4 k(1-\xi_k)^2 +O\bigl(\log^3n\bigr),
\end{align*}
implying a bound
\begin{equation}\label{6.85}
\begin{aligned}
h_k(s,s_1,\xi_k)&\le h_k(s,s_1,1)-0.4k(1-\xi_k)^2+O\bigl(\log^3n\bigr)\\
&=\frac{s_1^2}{2}-s^2-m\log s-0.4k(1-\xi_k)^2+O\bigl(\log^3n\bigr)\\
&\le -\frac{s^2}{2}-m\log s-0.4k(1-\xi_k)^2+O\bigl(\log^3n\bigr).
\end{aligned}
\end{equation}
We hasten to add that here $\xi_k$ still depends on $s$ and $s_1$. However
\begin{align*}
1-\xi_{k_1}(s,s_1)&=1-\frac{2k_1}{n^{3/2}}\bigl(1+O(n^{-1/2}\log^{1/2} n)\bigr)\\
&=1-\frac{2k_1}{n^{3/2}} +O(n^{-1/2}\log^{1/2} n)\\
&\ge n^{-1/2}\log n+O(n^{-1/2}\log^{1/2}n)\ge 0.9 n^{-1/2}\log n, 
\end{align*}
so $\xi_{k_1}(s,s_1)<1$, in particular. Similarly, $\xi_{k_2}(s,s_1)-1\ge 0.9 n^{-1/2}\log n$. Therefore
\[
h_{k_i} \bigl(s,s_1,\xi_{k_i}(s,s_1)\bigr)\le -\frac{s^2}{2}-m\log s-\frac{1}{6}n^{1/2}\log^2n, \quad (i=1,2).
\]
Having gained this extra term $-\frac{1}{6}n^{1/2}\log^2n$, we extend the integration with respect to $\bold x$ from $D^*$ to the whole $D$ and obtain
\begin{align*}
\max\bigl\{\Bbb P_s^*(\le k_1;\Pi),\,\Bbb P_s^*(\ge k_2;\Pi)\bigr\}&\le \frac{\exp(-\tfrac{1}{6}n^{1/2}\log^2n)}{(n-1)!}\int_0^{\infty} s^{l-1} e^{-s^2/2}\,ds\\
&=\frac{\exp(-\tfrac{1}{6}n^{1/2}\log^2n)(l-2)!!}{(n-1)!}.
\end{align*}
So (with $m:=n-l$) the expected number of all, but collectively negligible, $\Pi$'s such that $R_s(\Pi)\ge k_2$ ($R_s(\Pi)\le k_1$, resp.) is of order
\begin{align*}
&\exp(-\tfrac{1}{6}n^{1/2}\log^2n)\sum_{l\ge n-\omega(n)}\frac{\binom{n}{m} m! (l-1)!! (l-2)!!}{(n-1)!}\\
&=\exp(-\tfrac{1}{6}n^{1/2}\log^2n)\sum_{l\ge n-\omega(n)}\frac{n}{l}\le \exp(-\tfrac{1}{6.1}n^{1/2}\log^2n).
\end{align*}
Therefore, with probability $\ge 1-o(1)$, we have
\[
\min_{\Pi}R_s(\Pi)\ge \lfloor 0.5 n^{3/2}-n\log n\rfloor,\quad \max_{\Pi}R_s(\Pi)\le \lfloor 0.5 n^{3/2}+n\log n\rfloor.
\]
{\bf (II)\/} Finally, for each agent $i$ from the fixed pairs $\Pi(i)=\Pi^{-1}(i)$, and with probability $1-o(1)$ there are at most $\omega(n)$ agents outside
the fixed pairs, whence $\max_{\Pi}|R_s(\Pi)-R_p(\Pi)|\le n (n- \mathcal L(\Pi_0))\le n\omega(n)$.  Therefore 
\[
\min_{\Pi}R_p(\Pi)\ge \lfloor 0.5 n^{3/2}-1.1n\log n\rfloor,\quad \max_{\Pi}R_p(\Pi)\le \lfloor 0.5 n^{3/2}+1.1n\log n\rfloor
\]
with probability $\ge 1-o(1)$ as well.
\end{proof}

{\bf Appendix.\/}  {\it Proof of the bound \eqref{6.714}.\/} Let $\bold L=\{L_1,\dots, L_n\}$ be the lengths of the consecutive subintervals of $[0,1]$
obtained by throwing in uniformly and independently $n-1$ points into $[0,1]$. Let $m:=n-l=o(n)\to\infty$, and let $\eps\in (0,1)$ be
fixed. We need to show that for large $n$
\[
\Bbb P\biggl(1-\biggl(\sum_{i=1}^{l}L_i\biggr)^2\ge \frac{2(1-\eps)}{1+\eps}\bigl(m/n-0.5(m/n)^2\bigr)\biggr)\ge 1 -1.1\, e^{mh(\eps)},
\]
where $h(\eps):=\log(1-\eps)+\eps$. We use the fact that $\bold L \overset{\Cal D}{\equiv}\biggl\{\frac{w_j}{W_n}\biggr\}_{i\in [n]}$,
where $w_1,\dots,w_n$ are independent exponentials and $W_n:=\sum_{i\in [n]}w_i$. If $\sum_{i=l+1}^n w_i\ge (1-\eps)m$
and $W_n\le (1+\eps)n$, then (since $(2-\eta)\eta$ increases on $[0,1]$) we have
\begin{multline*}
1-\biggl(\sum_{i=1}^{l}L_i\biggr)^2\overset{\Cal D}{\equiv}\biggl(2-\frac{\sum_{i=l+1}^n w_i}{W_n}\biggr)\frac{\sum_{i=l+1}^n w_i}{W_n}\\
\ge \biggl(2-\frac{(1-\eps)m}{(1+\eps)n}\biggr)\frac{(1-\eps)m}{(1+\eps)n}\\
\ge \frac{2(1-\eps)}{1+\eps}\bigl(m/n- 0.5(m/n)^2\bigr).
\end{multline*}
Now using $\Bbb E[e^{zw}]= (1-z)^{-1}$, $(z<1)$, and Chernoff-type bounds
\[
\Bbb P\biggl(\sum_{j\in [\nu]} w_j \le a\biggr)\le \min_{z\le 0}\frac{E^{\nu}[e^{zw}]}{e^{a\nu z}},\quad
\Bbb P\biggl(\sum_{j\in [\nu]} w_j \ge b\biggr)\le \min_{z\ge 0}\frac{E^{\nu}[e^{zw}]}{e^{b\nu z}},
\]
we obtain
\begin{align*}
\Bbb P\biggl(\sum_{i=l+1}^n w_i\le (1-\eps)m\biggr)\le e^{m\, h(\eps)},\quad\Bbb P\biggl(\sum_{i=1}^n w_i\ge (1+\eps)n\biggr)
\le e^{n\, h(-\eps)}.
\end{align*}
It remains to notice that
\begin{multline*}
\Bbb P\biggl(\sum_{i=l+1}^n w_i\ge (1-\eps)m;\,\,\sum_{i=1}^n w_i\le (1+\eps)n\biggr)\\
\ge \Bbb P\biggl(\sum_{i=l+1}^n w_i\ge (1-\eps)m\biggr)-\Bbb P\biggl(\sum_{i=1}^n w_i\ge (1+\eps)n\biggr)\\
\ge 1- e^{m\, h(\eps)} -e^{n\, h(-\eps)}\ge  1 -1.1\, e^{mh(\eps)}.
\end{multline*}
\end{document}